\documentclass[11pt]{article}
\usepackage{amsmath,enumerate,amsfonts,color,amssymb,amsthm,ulem}

\usepackage{tikz}

\def\NN{{\mathbb N}}

\def\RR{{\mathbb R}}

\def\mcA{{\mycal A}}
\def\mcB{{\mycal B}}

\def\mcF{{\mycal F}}

\def\eps{{\varepsilon}}

\newtheorem{theorem}{Theorem}
\newtheorem{proposition}{Proposition}
\newtheorem{lemma}{Lemma}
\newtheorem{definition}{Definition}
\newtheorem{remark}{Remark}

\newcounter{marnote}

\DeclareFontFamily{OT1}{rsfs}{}
\DeclareFontShape{OT1}{rsfs}{m}{n}{ <-7> rsfs5 <7-10> rsfs7 <10-> rsfs10}{}
\DeclareMathAlphabet{\mycal}{OT1}{rsfs}{m}{n}

\def\Id{{\rm Id}}

\def\be{\begin{equation}}
\def\ee{\end{equation}}

\def\Id{{\rm Id}}

\newcommand{\R}{\mathbb{R}}

\def\mn{{\mathring{n}}}

\def\be{\begin{equation}}
\def\ee{\end{equation}}
\def\bea#1\eea{\begin{align}#1\end{align}}
\def\non{\nonumber}

\def\softness{0.4}

\def\newsoftness{0.7}
\definecolor{softred}{rgb}{1,\softness,\softness}
\definecolor{softgreen}{rgb}{\softness,1,\newsoftness}
\definecolor{softblue}{rgb}{\softness,\softness,1}

\definecolor{softrg}{rgb}{1,1,\softness}
\definecolor{softrb}{rgb}{1,\softness,1}
\definecolor{softgb}{rgb}{\softness,1,1}

\definecolor{gray}{rgb}{0.5,0.5,0.7}
\definecolor{ddgreen}{rgb}{0,0.4,0.4}
\definecolor{magenta}{rgb}{1,0,1}
\definecolor{dblue}{rgb}{0,0,0.5}
\definecolor{ddcyan}{rgb}{0,0.4,0.9}
\definecolor{dmagenta}{rgb}{0.8,0,0.8}
\definecolor{dyellow}{rgb}{0.8,0.8,0}
\definecolor{gray}{rgb}{0.5,0.5,0.7}
\definecolor{vargreen}{rgb}{0,0.4,0.4}
\definecolor{dgreen}{rgb}{0,0.4,0.2}
\definecolor{dred}{rgb}{.8,0,0}
\definecolor{dmagenta}{rgb}{0.8,0,0.8}
\definecolor{dyellow}{rgb}{0.8,0.8,0}
\definecolor{dbrown}{rgb}{.5,0,0}

\definecolor{gray}{rgb}{0.5,0.5,0.7}
\definecolor{ddgreen}{rgb}{0,0.4,0.4}
\definecolor{dgreen}{rgb}{0,0.65,0}
\definecolor{magenta}{rgb}{1,0,1}
\definecolor{dblue}{rgb}{0,0,0.7}
\definecolor{dcyan}{rgb}{0,0.7,0.7}
\definecolor{dmagenta}{rgb}{0.8,0,0.8}
\definecolor{ddmagenta}{rgb}{0.7,0,0.9}
\definecolor{violet}{rgb}{0.4,0,0.9}
\definecolor{dyellow}{rgb}{0.8,0.8,0}
\definecolor{softblue}{rgb}{\softness,\softness,1}

\newtheorem{teor}{Theorem}[section]
\newtheorem{defin}[teor]{Definition}
\newtheorem{lemm}[teor]{Lemma}
\newtheorem{osse}[teor]{Remark}
\newtheorem{prop}[teor]{Proposition}
\newtheorem{defi}[teor]{Definition}
\newtheorem{coro}[teor]{Corollary}
\newtheorem{prob}[teor]{Problem}
\newtheorem{assu}[teor]{Assumption}

\newcommand{\bele}{\begin{lemm}\begin{sl}}
\newcommand{\enle}{\end{sl}\end{lemm}}
\newcommand{\bedef}{\begin{defi}\begin{sl}}
\newcommand{\eddef}{\end{sl}\end{defi}}
\newcommand{\bete}{\begin{teor}\begin{sl}}
\newcommand{\ente}{\end{sl}\end{teor}}
\newcommand{\beos}{\begin{osse}\begin{rm}}
\newcommand{\eddos}{\end{rm}\end{osse}}
\newcommand{\beas}{\begin{assu}\begin{rm}}
\newcommand{\eddas}{\end{rm}\end{assu}}
\newcommand{\bepr}{\begin{prop}\begin{sl}}
\newcommand{\empr}{\end{sl}\end{prop}}
\newcommand{\bepro}{\begin{prob}\begin{rm}}
\newcommand{\empro}{\end{rm}\end{prob}}
\newcommand{\bede}{\begin{defin}\begin{sl}}
\newcommand{\edde}{\end{sl}\end{defin}}
\newcommand{\beco}{\begin{coro}\begin{sl}}
\newcommand{\enco}{\end{sl}\end{coro}}

\newcommand{\de}{\partial}

\newcommand{\beeq}[1]{\begin{equation}\label{#1}}
\newcommand{\eddeq}{\end{equation}}

\newcommand{\beeqa}[1]{\begin{eqnarray}\label{#1}}
\newcommand{\eddeqa}{\end{eqnarray}}

\newcommand{\beal}[1]{\begin{align}\label{#1}}
\newcommand{\eddal}{\end{align}}

\newcommand{\bespl}[1]{\begin{split}\label{#1}}
\newcommand{\edspl}{\end{split}}

\newcommand{\bega}[1]{\begin{gather}\label{#1}}
\newcommand{\edga}{\end{gather}}

\newcommand{\beeqax}{\begin{eqnarray*}}
\newcommand{\eddeqax}{\end{eqnarray*}}

\def\qed{\ifmmode 
  \else \leavevmode\unskip\penalty9999 \hbox{}\nobreak\hfill
  \fi
  \quad\hbox{\hskip.5em\vrule width.4em height.6em depth.05em\hskip.1em}}
\def\endproofsym{\qed}
\renewenvironment{proof}[1][Proof]{\trivlist\item[\hskip\labelsep{\hskip0pt
    {\normalfont\scshape#1.}\hskip .321429\parindent}]\ignorespaces}
{\endproofsym\endtrivlist}
\def\endnobox{\def\endproofsym{}\end{proof}\def\endproofsym{\qed}}

\newcommand{\no}{\nonumber}

\newcommand{\beeqao}{\begin{eqnarray}\no}
\newcommand{\bealo}{\begin{align}\no}
\newcommand{\besplo}{\begin{split}\no}
\newcommand{\begao}{\begin{gather}\no}

\newcommand{\ov}{\overline}

\newcommand{\io}{\int_{\mathcal{T}^3}}

\newcommand{\oo}{_{\Omega}}

\newcommand{\tor}{\mathcal{T}^3}

\def\R{\mathbb R}

\newcommand{\lhs}{left-hand side}
\newcommand{\rhs}{right-hand side}

\DeclareMathOperator{\dive}{div}
\DeclareMathOperator{\deriv}{d}

\DeclareMathOperator{\sign}{sign}

\let\TeXchi\chi
\def\chi{{\setbox0 \hbox{\mathsurround0pt
$\TeXchi$}\hbox{\raise\dp0 \copy0 }}}

\newcommand{\calM}{{\mathcal M}}

\newcommand{\dit}{\deriv\!t}

\newcommand{\ddt}{\frac{\deriv\!{}}{\dit}}

\newenvironment{giuliorev}{\color{red}}{\color{black}}
\newcommand{\III}{\begin{giuliorev}}
\newcommand{\EEE}{\end{giuliorev}}

\newenvironment{giuliorevb}{\color{magenta}}{\color{black}}
\newcommand{\AAA}{\begin{giuliorevb}}
\newcommand{\BBB}{\end{giuliorevb}}

\newenvironment{giuliorevc}{\color{blue}}{\color{black}}
\newcommand{\CCC}{\begin{giuliorevc}}
\newcommand{\DDD}{\end{giuliorevc}}

\newenvironment{bettirevddg}{\color{ddgreen}}{\color{black}}
\newcommand{\FFF}{\begin{bettirevddg}}
\newcommand{\GGG}{\end{bettirevddg}}

\newenvironment{bettirevdc}{\color{dcyan}}{\color{black}}
\newcommand{\QQQ}{\begin{bettirevdc}}
\newcommand{\KKK}{\end{bettirevdc}}

\newenvironment{bettirevg}{\color{green}}{\color{black}}
\newcommand{\MMM}{\begin{bettirevg}}
\newcommand{\NNN}{\end{bettirevg}}

\newenvironment{rev28}{\color{orange}}{\color{black}}

\numberwithin{equation}{section}

\begin{document}
\title{Nonlinear electrokinetics in nematic electrolytes}

\author{Eduard Feireisl\thanks{Institute of Mathematics of the Academy of Sciences of the Czech Republic, \v Zitn\' a 25, 
    CZ-115 67 Praha 1, Czech Republic {\it (feireisl@math.cas.cz)}}~, Elisabetta Rocca\thanks{Universit\`a degli Studi di Pavia, Dipartimento di Matematica, and IMATI-C.N.R, Via Ferrata 5, 27100,
Pavia, Italy {\it (elisabetta.rocca@unipv.it).}}~,\\ Giulio Schimperna\thanks{Universit\`a degli Studi di Pavia, Dipartimento di Matematica, Via Ferrata 5, 27100, Pavia, Italy
{\it (giusch04@unipv.it)}}~ and Arghir Zarnescu\thanks{IKERBASQUE, Basque Foundation for Science, Maria Diaz de Haro 3,
48013, Bilbao, Bizkaia, Spain}\,  \thanks{BCAM,  Basque  Center  for  Applied  Mathematics,  Mazarredo  14,  E48009  Bilbao,  Bizkaia,  Spain
{\it (azarnescu@bcamath.org)}}\, \thanks{``Simion Stoilow" Institute of the Romanian Academy, 21 Calea Grivi\c{t}ei, 010702 Bucharest, Romania}}

\maketitle

\begin{abstract}

In this article we study a system of nonlinear PDEs modelling the electrokinetics of a nematic electrolyte material consisting of various ions species contained in a nematic liquid crystal. 

The evolution  is described by a system coupling a Nernst-Planck system for the ions concentrations with a Maxwell's equation of electrostatics 
governing the evolution of the electrostatic potential, a Navier-Stokes equation for the velocity field, and 
a non-smooth Allen-Cahn type equation for the nematic director field.

We focus on the two-species case and prove apriori estimates that  provide a weak sequential stability result, the main step towards proving the existence of weak solutions.

\end{abstract}


\section{Introduction}

In this paper we consider a version of the system derived in \cite[(2.51)-(2.55)]{electrolytes} describing
the electrokinetics of a nematic electrolyte that consists of ions that diffuse and advect in a 
nematic liquid crystal environment.

 The system can be written in terms of the following variables:
 \begin{itemize}
 \item the vector $n$  modelling the local orientation of the nematic liquid crystal molecules,
 \item the macroscopic velocity $v$ of the liquid crystal molecules, 
 \item the pressure $p$ resulting from   the incompressibility constraint on the fluid,
 \item the electrostatic potential $\Phi$,
\item the concentrations $c_k$, $k=1,\dots,N$, with valences $z_k\in \{-1,1\}$, of the families of charged ions present in the liquid crystal. 
\end{itemize}
Actually, we consider a modified version of the system in \cite{electrolytes},  assuming certain  simplifications commonly used in the mathematical literature on liquid crystals. More specifically we take equal elastic constants in the Oseen-Frank energy 
and use a Ginzburg-Landau configuration potential $\mcF$ of {\it singular}\/ type
(see below for more details) in order to avoid introducing  the unit length constraint
(cf.~equation $(2.56)$ of \cite{electrolytes}) on $n$ (and thus we can correspondingly drop the 
related Lagrange multiplier term $\lambda n$ in the system in \cite{electrolytes}). Furthermore 
neglecting body forces and inertial effects acting on the director field, we can
write the resulting PDE system as follows:
\begin{align}
  \frac{\partial c_k}{\partial t}+v\cdot\nabla c_k
    & = \frac{1}{k_B\theta} \dive\left(c_k\mathcal{D}_k \nabla\mu_k\right), \quad\textrm{for }k=1,\dots,N, \label{eq:c_k-}\\
  -\dive(\eps_0\eps(n)\nabla\Phi) & = \sum_{k=1}^N qz_kc_k,\label{eq:Phi-}\\
  \frac{\partial v}{\partial t}+(v\cdot\nabla) v+\nabla p 
    & = -K \dive(\nabla n\odot \nabla n) + \dive\sigma,\non\\
  & \mbox{}~~~~~
   +\eps_0\dive\big( (\nabla\Phi\otimes\nabla\Phi)\eps (n)\big), \label{eq:v-}\\
  \dive v & = 0,\label{eq:vincompres-}\\
  \gamma_1(n_t+v\cdot\nabla n-\Omega(v) n)+\gamma_2 D(v)n
   & = K\Delta n+\eps_0\eps_a \left(\nabla\Phi\otimes \nabla \Phi\right)n-\partial\mcF,\label{eq:n-}
\end{align} 
where $\mu_k$ are the  electrochemical potentials of the ions  associated to the various ions species $c_k$, given by
\be
  \mu_k:= k_B\theta(\ln (c_k)+1) + qz_k\Phi,
\ee
$k_B>0$ denotes  the Boltzmann constant, $\theta>0$ stands for the absolute temperature, and $q$ denotes the elementary charge. 

Moreover, we have indicated by 
\be\label{def:DO}
  D(v):=\frac{1}{2}(\nabla v+\nabla v^t) \quad\text{and }\,\Omega(v):=\frac{1}{2}(\nabla v-\nabla v^t)
\ee
the symmetric and antisymmetric parts of the velocity gradient.
The diffusion operator in \eqref{eq:Phi-} is ruled by the 
matrix
\be
  \eps(n):=\eps_\perp \Id + \eps_a n\otimes n,
\ee
with constants $\eps_\perp>0$ and $\eps_a\ge 0$, $\Id$ denoting the identity matrix.  Here $\eps_a=\eps_{||}-\eps_\perp$, where $\eps_{||}$ and $\eps_\perp $ denote the electric permittivity when the electric field ${\bf E}=\nabla \Phi$ is parallel, respectively, perpendicular to $n$.

The constant $\eps_0>0$ stands for the vacuum
dielectric permeability. The matrices ${\mathcal D}_k$ are positive 
definite, i.e., 
\be\label{ass:Dkpossdef}
({\mathcal D}_k \xi) \cdot \xi > \alpha |\xi|^2
\ee for some $\alpha>0$
and all $k=1,\dots,N$ and $\xi\in \RR^3$.
In the above we have denoted by $\nabla n\odot \nabla n$ the $3 \times 3$ matrix whose $(i,j)$-component is 
$n_{k,i}n_{k,j}$ (here and in the sequel we assume summation over repeated indices). As customary,
for $a,b\in\R^3$ we denote as $a\otimes b$ the $3\times3$ matrix with component $(i,j)$ given by $a_i b_j$. 
 We will further assume that the system is non-dimensionalized, so the constants are dimensionless (this can be achieved similarly as in Section $3.2$ in  \cite{electrolytes}).

The Nernst-Planck type equations \eqref{eq:c_k-} correspond  to the  continuity equation for ions
with the electric potential $\Phi$ satisfying the Maxwell's equation of electrostatics \eqref{eq:Phi-}.

The Navier-Stokes equations \eqref{eq:v-}, with the incompressibility constraint \eqref{eq:vincompres-},
rule the evolution of the liquid crystal flow. Note the Korteweg forces on the \rhs\ being 
induced by the the director field $n$  and the effects of the electric field, respectively. 
As in \cite{leslie92}, we assume for the total stress tensor the following general expression:
\be\label{def:sigma}
  \sigma=\alpha_1(D(v)n\cdot n)n\otimes n+\alpha_2 \mn\otimes n+\alpha_3 n\otimes\mn+\alpha_4 D(v)+\alpha_5 D(v)n\otimes n+\alpha_6 n\otimes D(v)n ,
\ee 
where we have denoted $\mn:= \partial_t n+v\cdot\nabla n-\Omega(v) n$ the Lie derivative of $n$. Here the term $\alpha_4 D(v)$ represents the classical Newtonian stress tensor, while the other terms represent the additional stress produced by the interaction of the anisotropic liquid crystal molecules, see \cite{ericksenfirst,ericksenpos}.

As mentioned above, we avoid to insert the unit length constraint in \eqref{eq:n-} and instead require $|n|\le 1$, in the spirit of the the variable length model proposed by J. L. Ericksen in \cite{ericksenvariable}. Indeed,
following an approach commonly used in the context of phase-transition models, we 
enforce the property $|n|\le 1$ by means of the {\it singular potential}\/ $\mcF$.
Namely, we assume $\mcF:\RR^3 \to [0,+\infty]$ be a convex and lower semicontinuous
function whose {\it effective domain}\/ (i.e., the set where it attains finite values)
is assumed to coincide with the closed unit ball $\ov B_1$ of $\RR^3$, with a reference 
choice being given by 
\be\label{singpot}
  \mcF(n) = \frac12 F(|n|^2),
\ee 
where $F$ is convex and has the interval $(-\infty,1]$ as an effective domain. We will actually choose $F(r) = (1-r) \log (1-r)$, 
an expression mutuated from the Cahn-Hilliard logarithmic potential, but we point out that more general choices may be allowed.

Such an idea was introduced by J.L. Ericksen in \cite{ericksenvariable} in order to enforce the physicality of a scalar order parameter and has already been applied to liquid crystal models 
in a number of papers  (cf., e.g., ~\cite{FRSZ1,FRSZ2}) 
and  has the advantage that  as soon as we have proved existence of  a solution,  then 
the constraint $|n|\leq 1$ is authomatically satisfied. This helps in the estimates which 
actually could not be performed in this way in the case of a {\sl classical} double-well potential. 

\smallskip

Finally, in order to avoid complications due to the interaction with the boundary, we will settle the above system
on the flat 3-dimensional torus
\begin{equation} \label{i5}
  \tor = \left( [-\pi, \pi]|_{\{ - \pi, \pi \}} \right)^3
\end{equation}
so assuming periodic boundary conditions. We note that more realistic choices for the boundary conditions could 
be likely taken. Nevertheless the above setting, beyond being the simplest one mathematically, is also consistent
with the basic physical principles of conservation of charge and of momentum 
(indeed, we assume no external forces be present), that can be verified respectively by integrating \eqref{eq:c_k-} and
\eqref{eq:v-} with respect to space variables.

\smallskip

Our main aim here is to set the ground for proving the existence of weak solutions. These are usually obtained via  three steps: {\it `apriori estimates'}, {\it `approximation scheme'}, and {\it `compactness'}. 

 The apriori estimates are obtained on  {\it presumptive} smooth solutions of the equation. Such   estimates  allow to control 
 (in terms of initial data and fixed parameters of the system)  certain norms, sufficiently strong, in order to allow to pass to the limit in the approximation scheme. 
 
 The approximation scheme is usually designed such that one can obtain estimates for the approximating equations  that are usually very close to the apriori estimates. 
 The construction of such a scheme can be a highly tedious and non-trivial issue in presence of complex systems as we consider (see for comparison 
 our previous works on non-isothermal liquid crystals, with an approximation scheme \cite{FRSZ1} and without one, just with apriori estimates as in here \cite{FRSZ2}). Thus
 we will leave the construction of such a scheme to interested readers and focus just on the first part, namely obtaining apriori estimates that are strong enough in order 
 to allow to pass to the limit in the approximation scheme  via compactness and we will refer to this as {`weak sequential stability'}, the main content of Theorem~\ref{thm:stab}.

In addition to that, we will focus on a simplified version of 
system \eqref{eq:c_k-}-\eqref{eq:n-}, complemented with the Cauchy conditions and with periodic
boundary conditions in three dimensions of space and with no restrictions on the magnitude 
of the initial data. The precise simplifications will be introduced in the next section, but
it is worth observing that, beyond setting some physical constants equal to one, the only effective reduction we are actually going to operate concerns the number of species $c_k$ which 
will be assumed to be equal to $2$. Namely, we only take two species $c_p$ and $c_m$, which will
then denote the density of positive and negative charges, respectively. Mathematically
speaking, this ansatz simplifies the nature of the system \eqref{eq:c_k-}-\eqref{eq:Phi-},
and in particular permits us to prove by means of very simple maximum principle arguments 
the uniform boundedness of $c_p$ and $c_m$, which is a key ingredient for obtaining
the apriori estimates. 

It is worth noting that we expect the same boundedness property to hold also in
the general case of $N$-species, however the proof may be much more involved 
and require use of more technical results about invariant regions for evolutionary systems
(see, e.g., \cite{CF}). We also expect that similar arguments could be applied  in the more complicated systems where 
one uses a tensorial order parameter,  that is a matrix valued function, i.e.~a $Q$-tensor in the LC terminology, instead of the vector-valued one, $n$,
as done for instance in \cite{Qelectrolytes}. The current work is related to work done in certain simpler systems 
that can be regarded as subsets of our equations, such as Nernst-Planck-Navier-Stokes  system (see for instance \cite{ConstantinNP}
 and the references therein) and liquid crystal equations (see for instance the review \cite{LinWangreview}).

\smallskip

The main ingredients of the proofs are the following: first we perform an energy estimate which is 
mainly based on a key Lemma (cf.~Lemma~\ref{lemma1}) providing sufficient conditions 
on the $\alpha_i$-coefficients such that the dissipation is non-negative. Then, via a maximum-principle technique, we prove 
pointwise bounds for $c_p$ and $c_m$. The $L^\infty$-estimate on the potential $\Phi$ follows instead by a Moser-iteration scheme 
proved in Lemma~\ref{lem:phi}, while in Lemma~\ref{lem:meyer} we state an $L^p$-regularity result for $n$. This result, based on an $L^p$-estimate for the potential $\partial \mcF$, 
is in general new in the framework of 
non-smooth parabolic systems, while it is quite known in case of scalar equations (cf., e.g., \cite{cgmr16}). Finally, an additional regularity result for $n$ (cf.~Lemma~\ref{lem:n}) is shown
in case the anisotropy coefficient $\eps$ is sufficiently small. In the last Section~\ref{sec:limit} the weak sequential stability property result is proved for every $\eps>0$.

The plan of the paper is as follows: in the next section~\ref{sec:main} we introduce the simplified 
version of system \eqref{eq:c_k-}-\eqref{eq:n-} and state the precise formulation of our existence
theorem. Then, the basic apriori estimates are derived in Section~\ref{sec:apriori}.

Finally, in Section~\ref{sec:limit} we will prove the stability result.


\section{Main results}
\label{sec:main}

We start introducing some notation. Given a space of functions defined over $\Omega=\tor$,
we will always use the same notation for scalar-, vector-, or tensor-valued function. 
For instance, we will  indicate by the same letter $H$ the spaces 
$L^2(\Omega)$, $L^2(\Omega)^3$ and $L^2(\Omega)^{3\times 3}$. Correspondingly,
the norm in $H$ will be simply denoted  by $\| \cdot \|$. The notation
actually subsumes the periodic boundary conditions. We also set $V=H^1(\Omega)$
(or $H^1(\Omega)^3$, or $H^1(\Omega)^{3\times 3}$). For two $3\times3$ matrices $A$, $B$, we also set $A:B:=A_{ij}B_{ij}$.

\bigskip

In view of the discussion carried out above, we now introduce the simplified system for which we shall
prove existence of weak solutions. Namely, we assume $\eps_\perp=k_B\theta=K=\eps_0=q=\gamma_1=\gamma_2=1$ 
and write $\eps$ in place of $\eps_a$. Moreover, we only take two species $c_p$ and $c_m$ with $z_p=1$ and $z_m=-1$. Moreover we take, similarly in spirit as in \cite{electrolytes}, Section $3.1$, the matrices $\mathcal{D}_p=\mathcal{D}_k={\rm Id}+\eps n\otimes n$.\footnote{ This simplification is not necessary for obtaining the energy law in Proposition~\ref{prop:energylaw}, but essential in deriving the maximum principle in Proposition~\ref{prop:maxprinc}  }
Then the simplified system takes the form
\bea
  \frac{\partial c_p}{\partial t}+v\cdot\nabla c_p & = \dive \big((\Id+\eps n\otimes n) (\nabla c_p+c_p\nabla \Phi)\big),\label{eq:cp}\\
  \frac{\partial c_m}{\partial t}+v\cdot\nabla c_m & = \dive \big ((\Id+\eps n\otimes n) (\nabla c_m-c_m\nabla \Phi)\big),\label{eq:cm}\\
  -\dive \big((\Id+\eps n\otimes n) \nabla\Phi\big) & = c_p-c_m,\label{eq:Phi}
\eea
\bea
  \frac{\partial v}{\partial t}+(v\cdot\nabla) v+\nabla p & = \alpha_4  \dive D(v)  -\dive(\nabla n\odot \nabla n)\non\\
  & \mbox{}~~~ + \dive \big( (\nabla\Phi\otimes\nabla\Phi) (\Id+\eps n\otimes n) \big)\non\\
  & \mbox{}~~~ + \dive\big(\alpha_1 (D(v)n\cdot n)n\otimes n+\alpha_2 \mn\otimes n+\alpha_3 n\otimes\mn\big)\non\\
  & \mbox{}~~~ + \dive\big(\alpha_5 D(v)n\otimes n+\alpha_6 n\otimes D(v)n\big),\label{eq:v}\\
  \dive v& = 0,\label{eq:vincompres}\\
  %
  n_t+v\cdot\nabla n-\Omega(v) n
  + D(v)n  
   & = \Delta n+\eps\left(\nabla\Phi\otimes \nabla \Phi\right)n-\partial\mcF(n).\label{eq:n}
\eea
Note that $\partial\mcF$ denotes the {\it subdifferential}\/ of $\mcF$ in the sense of convex
analysis. Although one can use more general assumptions on the potential here we are assuming for definiteness that
\be\label{hp:mcF}
  \mcF(n):=\left\{\begin{array}{ll}\frac12 F(|n|^2)-F_*,& \textrm{ if }|n|\leq1\\
  +\infty, \textrm{ otherwise }\end{array} \right.
\ee where 

\be\label{hp:F}
F(r):=(1-r)\log (1-r)-F_*, \,\, r\in (0,1),
\ee
and $F_*$ is chosen such that $\min \,F(r)=F(1-1/e)=0$.

Moreover, in order to prove the energy estimate (cf.~Lemma~\ref{lemma1}), let us suppose  that there exists $\delta>0$ such that 

\be\label{hp:alpha}
  \alpha_4>0, \quad  \alpha_4-|\alpha_1|-|\alpha_5|-|\alpha_6|-\frac{1}{1-\delta} > 0.
  \ee
Finally, we assume the initial data to satisfy the following conditions,
where $\bar c>0$ is a given constant:
\bea\label{hp:init}
  & c_{p,0},c_{m,0} \in L^\infty(\tor), \quad
   0 \le c_{p,0},c_{m,0} \le {\bar c}~\,\text{a.e.~in }\,\tor,\\
 \label{hp:v0}
  & v_0 \in L^2(\tor), \quad
   \dive v_0 = 0,\\
 \label{hp:n0}
  & n_0 \in H^1(\tor), \quad |n_0(x)|\le 1, \forall x\in \tor,
\eea

\bigskip
Let us now define the weak solutions, in a rather standard way, but emphasizing the spaces of functions used.
\begin{definition}\label{def:weak} [Weak solutions]
Assume hypotheses \eqref{hp:F}, \eqref{hp:init}--\eqref{hp:n0}. Then, the functions
 \bea\label{re:v}
 & v \in L^\infty(0,T;H) \cap L^2(0,T;V),\\
 \label{re:n}
  & n\in  W^{1,p_0}(0,T; L^{p_0}(\tor))\cap L^{p_0}(0,T; W^{2,p_0}(\tor))\cap L^\infty(0,T;V) \cap L^\infty((0,T)\times\tor),\\
  \label{re:f}
    &  \mcF (n)\in L^{p_0}(0,T;L^{p_0}(\tor))\quad\hbox{for some }p_0>1,\\
      \label{re:phi}
  & \Phi \in L^\infty(0,T;V) \cap L^\infty(0,T;L^\infty(\tor))\cap  L^\infty(0, T; W^{1,p_M}(\tor))\quad\hbox{for some }p_M>2, \\
 \label{re:c}
  & c_p,~c_m \in W^{1, 4/3}(0,T; V') \cap L^2(0,T;V) \cap L^\infty(0,T;L^\infty(\tor)),\\
  \label{re:c2}
  & c_p, \, c_m\geq 0 \text{ a.e.~in } \tor\times (0,T),
 \eea
are a weak solution of \eqref{eq:cp}--\eqref{eq:vincompres} provided that
 \bea
 & \int_0^T\left(<\frac{\partial c_p}{\partial t}, \phi_p>{-\int_\Omega vc_p\cdot\nabla \phi_p} \right) = \int_0^T\int_\Omega \big((\Id+\eps n\otimes n) (\nabla c_p+c_p\nabla \Phi)\big)\nabla\phi_p,\label{eq:cpw}\\
 &  \int_0^T\left(<\frac{\partial c_m}{\partial t}, \phi_m>{-\int_\Omega vc_m\cdot\nabla \phi_m\,dx} \right)  = \int_0^T\int_\Omega \big((\Id+\eps n\otimes n) ({\nabla c_m-c_m\nabla \Phi)}\big)\nabla\phi_m,\label{eq:cmw}\\
& \int_0^T\int_\Omega  \big((\Id+\eps n\otimes n) \nabla\Phi\big):\nabla u =\int_0^T\int_\Omega(c_p-c_m) u,\label{eq:Phiw}\\  
& \int_0^T\int_\Omega v\frac{\partial z}{\partial t}+{(v\otimes v)}:\nabla z    = -\int_\Omega v_0 z(0)\,dx+\int_0^T\int_\Omega\sigma:\nabla z-(\nabla n\odot \nabla n):\nabla z\non \\
  & \qquad\qquad  \qquad\qquad \qquad\qquad\qquad+ \int_0^T\int_\Omega\big( (\nabla\Phi\otimes\nabla\Phi) (\Id+\eps n\otimes n) \big):\nabla z\label{eq:vw}\\
 %
  %
 & n_t+v\cdot\nabla n-\Omega(v) n
  + D(v)n  
   = \Delta n+\eps\left(\nabla\Phi\otimes \nabla \Phi\right)n-\partial\mcF(n)\quad \hbox{a.e.~in } \tor\times (0,T),\label{eq:nw}
\eea 
with $\sigma$, $D(v)$, and $\Omega(v)$ defined as in \eqref{def:sigma} and \eqref{def:DO}, and holding true for every test functions $\phi_p,\,\phi_m\in L^4(0,T; V)$, $u\in L^2(0,T;V)$,
$z\in C^\infty(\tor\times [0,T])$, {$\dive z=0$}   and coupled with the initial conditions:
\be\label{init}
  c_p(0) = c_{p,0},\,\,
  c_m(0) = c_{m,0},  \quad \hbox{in }V',\qquad
  n(0) = n_{0}, \,\,
  v(0)=v_0,\quad\hbox{a.e.~in }\tor.
\ee

\end{definition}

The {\it weak sequential stability} theorem we aim to prove is the following:

\begin{theorem}\label{thm:stab}
Let us assume that there exists a family $(c_{p}^{(k)},c_{m}^{(k)},\Phi^{(k)},v^{(k)},n^{(k)})_{k\in\mathbb{N}}$ of smooth solutions of the system  
\eqref{eq:cp}--\eqref{eq:vincompres} on the flat 3-dimensional torus $\tor$  subject to corresponding initial data
\be\label{init:a}
  c_{p}^{(k)}(0) = c_{p,0}^{(k)}, \quad
  c_{m}^{(k)}(0) = c_{m,0}^{(k)}, \quad
  n^{(k)} (0)= n_{0}^{(k)},
\ee with $(c_{p,0}^{(k)},c_{m,0}^{(k)},n_{0}^{(k)})\in (C^\infty(\tor))^3$.
We furthermore assume that the conditions \eqref{hp:mcF},\eqref{hp:F}, \eqref{hp:init}--\eqref{hp:n0}, \eqref{ass:Dkpossdef} hold.
Moreover  we assume that there exists a constant $\tilde C$, independent of $k\in\mathbb{N}$, such that
\be
\|c_{p,0}^{(k)}\|_{L^\infty},\|c_{m,0}^{(k)}\|_{L^\infty},\|n_{0}^{(k)}\|_{H^1(\tor)}, \|v_0\|_H\le \tilde C\quad\hbox{and } c_{i,0}^k\to c_{i,0}, \quad n_0^k\to n_0,
\ee
the latter convergence relations holding, e.g., in the sense of distributions. 

\medskip
 Then there exists a (non-relabelled) sequence of the family $(c_{p}^{(k)},c_{m}^{(k)},\Phi^{(k)},v^{(k)},n^{(k)})$
tending, in the sense explicated in relations \eqref{li:11}--\eqref{li:17} below, to a quintuple $(c_{p},c_{m},\Phi,v,n)$ solving
system~\eqref{eq:cp}-\eqref{eq:n} in the sense specified in Definition~\ref{def:weak}.
\end{theorem}
\begin{remark}
In fact one would need solutions which are not smooth but just `sufficiently regular', but the precise minimal regularity 
needed is not of interest since in general the solutions obtained through approximations scheme are smooth.
\end{remark}

The rest of the paper is devoted to the proof of Theorem~\ref{thm:stab}.


\section{Apriori estimates}
\label{sec:apriori}

We now prove a number of apriori estimates on the solutions of system~\eqref{eq:cp}-\eqref{eq:n}. As
noted above, we decided to perform the computations by directly working on the ``limit'' equations
without referring to any explicit regularization or approximation scheme. Of course, in such
a setting, the procedure has just a formal character because the use of some test function
as well as some integration by parts is not justified (this, for instance, surely happens in
connection with the Navier-Stokes system~\eqref{eq:v}). On the other hand, the computations
we are going to develop are not trivial and involve a certain number of subtlenesses; for
this reason we believe that presenting them in the simplest possible setting might
help comprehension. Actually, in the last part of the paper we will provide some 
hints about the construction of an approximation
scheme being compatible with the estimates. 

The first property we prove is the basic energy estimate resulting as a consequence of the variational
nature of the model. We state it in the form of a 
\begin{proposition}[Energy law]\label{prop:energylaw}
 Let $(c_m,c_p,\Phi,v,n): \Omega\to \RR\times\RR\times\RR\times \RR^3\times \RR^3$ be a sufficiently 
 smooth solution of system~\eqref{eq:cp}-\eqref{eq:n} on $\tor\times (0,T)$ complemented with 
 the initial conditions~\eqref{init} and satisfying the coefficient relations \eqref{hp:alpha}(that ensure the non-negativity  of the dissipation).
 Then there holds the energy inequality
 \bea
   & E(t) + \int_0^t \io \bigg( \frac{1}{c_p} |\nabla c_p+c_p\nabla\Phi|^2 +  \frac{1}{c_m} |\nabla c_m-c_m\nabla\Phi|^2 \bigg)\\
   & \mbox{}~~~~~~~~~ +\int_0^t\io \big(  \underbrace{\alpha_4|D(v)|^2
   +\alpha_1(n\cdot D(v)n)^2+2 (\mathring{n}\cdot D(v)n)+(\alpha_5+\alpha_6)|D(v)n|^2+|\mathring{n}|^2}_{\ge 0 } \big)\non\\
 \label{est:en} 
  & \mbox{}~~~~ \le E(0)
\eea 
where the energy functional is defined as 
\be\label{defi:en} 
  E(t)=\io \Big( \frac{1}{2}|v|^2+\frac{1}{2}|\nabla n|^2+\mcF(n)+c_p\ln c_p+c_m\ln c_m+\frac{1}{2}(1+\eps n\otimes n)\nabla \Phi\cdot\nabla\Phi \Big).
\ee
\end{proposition}
\begin{proof}
We multiply the equation \eqref{eq:cp} by $\ln c_p+\Phi$, integrate by parts using periodic boundary conditions to obtain
\bea\non
  & \ddt \io c_p(\ln c_p-1)
  + \io c_p'\Phi
  + \io (v\cdot\nabla c_p)\Phi\\
 \label{cp:11} 
  & \mbox{}~~~~~ 
  + \io (\Id+\eps n\otimes n)(\nabla c_p+c_p\nabla\Phi)\cdot\Big(\frac{\nabla c_p}{c_p}+\nabla\Phi\Big)=0,
\eea 
whence, by positive definiteness of the matrix $n\otimes n$,
\be\label{cp:11b}
  \ddt \io c_p(\ln c_p-1)
   \underbrace{+\io c_p'\Phi}_{:=\mcA_{11}}+\underbrace{\io (v\cdot\nabla c_p)\Phi}_{:=\mcA_{21}}
    +\io \frac{1}{c_p}|\nabla c_p+c_p\nabla\Phi|^2\le 0.
\ee
Similarly, testing \eqref{eq:cm} by $\ln c_m-\Phi$ we have
\bea\non
  & \ddt \io c_m(\ln c_m-1)
  - \io c_m'\Phi
  - \io (v\cdot\nabla c_p)\Phi\\
 \label{cm:11} 
  & \mbox{}~~~~~ 
  + \io (\Id+\eps n\otimes n)(\nabla c_m - c_m\nabla\Phi)\cdot\Big(\frac{\nabla c_m}{c_m} - \nabla\Phi\Big) = 0,
\eea 
whence
\be\label{cm:11b}
 \ddt \io c_m(\ln c_m-1)
  \underbrace{-\io c_m'\Phi}_{:=\mcA_{12}}\underbrace{-\io (v\cdot\nabla c_m)\Phi}_{:=\mcA_{22}}
   +\io \frac{1}{c_m}|\nabla c_m-c_m\nabla\Phi|^2\le 0.
\ee
We now test \eqref{eq:Phi} by $-\partial_t\Phi$ getting, after an integration by parts,
\be\label{Phi:11}
  -\io (\Id+\eps n\otimes n) \nabla\Phi \cdot\nabla \Phi_t+\io (c_p-c_m)\Phi_t = 0,
\ee
which can be expanded into
\bea\non
  & \underbrace{-\frac{\eps}{2}\ddt \io (n\otimes n\nabla\Phi)\cdot\nabla\Phi-\io\nabla\Phi\cdot\nabla\Phi_t+\io (c_p-c_m)\Phi_t}_{:=\mcA_{13}}\\
 \label{Phi:12} 
  & \mbox{}~~~~~
  =\underbrace{-\frac{\eps}{2} \io \partial_t(n\otimes n)\nabla\Phi\cdot\nabla\Phi}_{:=\mcA_3}.
\eea
Multiplying  \eqref{eq:Phi}  by $-v\cdot\nabla\Phi$ and integrating by parts we get
\be\label{Phi:21}
  -\io\big((\Id+\eps n\otimes n)\nabla\Phi\big)\cdot \nabla(v\cdot\nabla\Phi)
   + \io (c_p-c_m)v\cdot\nabla\Phi=0.
\ee
Splitting the left-hand side and integrating by parts further, we obtain
\be\label{Phi:22}
  \underbrace{-\io (\nabla\Phi\otimes\nabla\Phi):\nabla v}_{:=\mcA_4}
   \underbrace{-\eps\io \big((n\otimes n)\nabla\Phi\big)\cdot�\nabla(v\cdot\nabla\Phi)}_{:=\mcA_5}
   \underbrace{+\io (c_p-c_m)v\cdot\nabla\Phi}_{:=\mcB_2}=0.
\ee
Multiplying \eqref{eq:v} by $v$ and integrating by parts, we get
\bea
  & \frac 12 \ddt \|v\|^2
   + \alpha_4\| D(v) \|^2 =
    \underbrace{\io (\nabla n\odot \nabla n):\nabla v}_{:=\mcA_6}
    \underbrace{-\io \nabla\Phi\otimes\nabla\Phi:\nabla v}_{\mcA_4}\non\\
  & \mbox{}~~~~~  
    \underbrace{-\eps\io \big( (\nabla\Phi\otimes\nabla\Phi) (n\otimes n) \big) :\nabla v}_{:=\mcB_{51}}\non\\
  & \mbox{}~~~~~
  - \io\underbrace{ \big(\alpha_1(Dn\cdot n)n\otimes n+\alpha_2 \mn\otimes n+\alpha_3 n\otimes\mn
         +\alpha_5 Dn\otimes n+\alpha_6 n\otimes Dn \big):\nabla v}_{:=\mcB_7}.\label{v:11}
\eea
Finally, multiplying \eqref{eq:n} by $\dot n = n_t+v\cdot \nabla n$ we get
\bea\non
   & \io\underbrace{ \big( \mn + D(v)n\big)\cdot \dot n}_{:=\mcA_7}
   + \frac{1}{2} \ddt \|\nabla n\|^2
   + \ddt\io \mcF(n)
   \underbrace{ +\io \nabla n\cdot \nabla(v\cdot \nabla n)}_{:=\mcB_6}\\
 \label{n:11}
  & \mbox{}~~~~~ 
   =\underbrace{\eps\io \nabla\Phi\otimes\nabla\Phi: n\otimes n_t}_{\mcB_3}
   \underbrace{ + \eps\io \big((\nabla\Phi\otimes\nabla\Phi) n\big)\cdot (v\cdot\nabla) n}_{\mcB_{52}}.
\eea
We can now sum \eqref{cp:11b}, \eqref{cm:11b}, \eqref{Phi:12}, \eqref{Phi:22}, \eqref{v:11}, \eqref{n:11}. 
We combine a number of terms and may note several cancellations, namely
\bea
  & \ddt\io \bigg( -\frac 12(\Id+\eps n\otimes n)\nabla\Phi \cdot \nabla\Phi
      +(c_p-c_m)\Phi \bigg) 
   =\mcA_{11}+\mcA_{12}+\mcA_{13},\non\\
 \non
  & \mcA_{21}+\mcA_{22}=\mcB_2, \qquad 
    \mcA_3 = \mcB_3, \qquad
   \mcA_6=\mcB_6.
\eea
The most delicate cancellation is $\mcA_5=\mcB_{51}+\mcB_{52}$, which amounts to
\be
  - \io (n\otimes n\nabla\Phi)\cdot \nabla (v\cdot\nabla\Phi)
   = - \io(\nabla\Phi\otimes\nabla\Phi)n\otimes n:\nabla v
    + \io (\nabla\Phi\otimes\nabla\Phi)n)\cdot (v\cdot\nabla)n, \non
\ee
which, after expanding $(n\otimes n\nabla\Phi)\cdot \nabla (v\cdot\nabla\Phi)
=(n\otimes n\nabla\Phi)\cdot (\nabla v\nabla\Phi)+ (n\otimes n \nabla\Phi)\cdot (v\cdot\nabla)\nabla\Phi$, simplifies to
\be\label{trickyintparts}
  - \io n_in_j\partial_j\Phi v_k\partial_i\partial_k \Phi
   = \io\partial_i\Phi\partial_j\Phi n_j v_k \partial_k n_i.
\ee
Then, we integrate by parts the $\partial_k $ derivative and note that no boundary terms 
appear due to the choice of periodic boundary conditions. Hence, using $\partial_k v_k=0$ 
we obtain
\bea\non
  & - \io n_in_j\partial_j\Phi v_k\partial_i\partial_k \Phi
   = \io n_{i,k}n_j \partial_j\Phi\partial_i\Phi v_k\\
  & \mbox{}~~~~~
   + \io n_in_{j,k} \partial_j\Phi\partial_i\Phi v_k
   + \io n_in_j \partial_j\partial_k\Phi\partial_i\Phi v_k.
\eea
We note that after permuting the indices the above turns into
\be
  -2\io n_in_j\partial_j\Phi v_k\partial_i\partial_k \Phi
   = 2\io n_{i,k}n_j \partial_j\Phi\partial_i\Phi v_k,
\ee
which is exactly \eqref{trickyintparts}, thus proving the claimed cancellation $\mcA_5=\mcB_{51}+\mcB_{52}$.

Furthermore, as in \cite{electrolytes} we have
\be
  \mcA_7+\mcB_7= \alpha_1(n\cdot Dn)^2 + 2 (\mathring{n}\cdot Dn)
   +\alpha_4|D|^2
   +(\alpha_5+\alpha_6)|Dn|^2
   + |\mathring{n}|^2
\ee
Collecting the above computations, and using also the charge conservation property
\be\label{charge}
  \ddt \io ( c_p + c_m ) = 0,
\ee
we finally arrive at 
\bea
  & \ddt \io \bigg(\frac 12 | v|^2 +\frac 12 |\nabla n|^2
   +\mcF(n) +c_p\ln c_p+c_m\ln c_m\non\\
  & \mbox{}~~~~~~~~~~~~~~~
   -\frac 12 (\Id+\eps n\otimes n)\nabla\Phi \cdot \nabla\Phi
   + (c_p - c_m ) \Phi  \bigg)\non\\
  & \mbox{}~~~~~
   + \io \bigg(  \frac{1}{c_p} |\nabla c_p+c_p\nabla\Phi|^2
   +  \frac{1}{c_m} |\nabla c_m-c_m\nabla\Phi|^2
   + \alpha_4| D(v) |^2  \bigg) \non\\
 \label{quasi:en}
  & \mbox{}~~~~~
    +\io \Big( \alpha_1(n\cdot Dn)^2
    + 2 (\mathring{n}\cdot Dn)^2
    + (\alpha_5+\alpha_6) |Dn|^2
    + |\mathring{n}|^2 \Big) 
  \le 0.
\eea
Let us now notice that, testing \eqref{eq:Phi} by $\Phi$ 
and integrating by parts, there follows
\be
  \io (c_p-c_m) \Phi
   = \io (\Id + \eps n\otimes n )\nabla\Phi \cdot \nabla\Phi.
\ee
Replacing the above into \eqref{quasi:en}, we obtain \eqref{est:en},
which concludes the proof.
\end{proof}
\noindent%
The energy estimate \eqref{est:en} implies a number of apriori bounds for the 
solutions of system \eqref{eq:cp}-\eqref{eq:n}, provided that the dissipation
term is nonnegative. In our simplified setting (where we have set $\gamma_1,\gamma_2=1$,
this results as a restriction on the choice of the parameters $\alpha_j$. Namely, we can observe the 
following
\begin{lemma}\label{lemma1}
If \eqref{hp:alpha} holds true, 
then we have, for some $\delta'>0$, 
\be\label{pos:diss}
\alpha_4|D|^2+  \alpha_1(n\cdot Dn)^2
   + 2 (\mathring{n}\cdot Dn)
   +(\alpha_5+\alpha_6)|Dn|^2
   + |\mathring{n}|^2\ge  \delta'\left( |Dn|^2
   + |\mathring{n}|^2\right)
\ee 
for arbitrary $\mn\in\mathbb{R}^3,n\in\mathbb{R}^3,\,D\in\mathbb{R}^{3\times 3}$ with $|n|\le 1$ and the matrix $D$ symmetric and traceless.
\end{lemma}
\begin{proof}
Noting that we have (where we use that $|n|\le 1$):
$$
(n\cdot Dn)^2\le |n|^2|Dn|^2\le |D|^2, \quad |2(\mathring{n}\cdot Dn)|\le 2|\mathring{n}||Dn|\le (1-\delta)|\mathring{n}|^2+\frac{1}{1-\delta}|D|^2
$$ we immediately deduce that \eqref{hp:alpha} implies the claimed \eqref{pos:diss}.
\end{proof}
\noindent%
In the sequel we shall always assume \eqref{hp:alpha}. In this way,
as a consequence of the energy estimate~\eqref{est:en}, using also the 
positive definiteness of the matrix $n\otimes n$ and  \eqref{hp:F}, we can obtain 
a number of apriori bounds holding for any hypothetical solution 
of the system and independently of any eventual approximation
or regularization parameter. Namely, we have
\bea\label{st:11}
  & \| v \|_{L^\infty(0,T;H)} + \| v \|_{L^2(0,T;V)} \le c,\\
 \label{st:12}
  & \| n \|_{L^\infty(0,T;V)} \le c, \qquad
   | n | \le 1 \quad\text{a.e.~in }\,(0,T)\times \tor,\\  
 \label{st:13}
  & c_p, c_m \ge 0 \quad\text{a.e.~in }\,(0,T)\times \tor,\\
 \label{st:14}
  & \| \nabla \Phi \|_{L^\infty(0,T;H)} \le c.
\eea 
where $c$ is a constant depeding only on $E(0)$ as defined in \eqref{defi:en}. 
Note that the second bound in \eqref{st:12} directly follows from our choice of the potential $F$. 
\begin{proposition}[Maximum principle]\label{prop:maxprinc}
 Let $c_p^0,c_m^0:\tor\to \R_+$ satisfy\/ \eqref{init} 
 and let $v,$ $n$ satisfy \eqref{st:11}, \eqref{st:12}. 
 Then, if $(c_p,c_m,\Phi)$ solve equations\/ \eqref{eq:cp}, \eqref{eq:cm}, \eqref{eq:Phi} 
 subject to periodic boundary conditions and initial data\/ $c_p^0,c_m^0$ 
 as above, then there follows 
 \be\label{mp}
   |c_p(x,t)|, |c_m(x,t)|\le \bar c, \quad\text{a.e.~in }\,(0,T)\times \tor.
 \ee
\end{proposition}
\begin{proof}
We multiply \eqref{eq:cp} by $(c_p-\bar c)^+$ and integrate over $\tor$ and by parts, to obtain
\bea
 & \frac{1}{2} \ddt \io |(c_p-\bar c)^+|^2
  + \frac12 \io v\cdot\nabla ((c_p-\bar c)^+)^2\non\\
 & \mbox{}~~~~~
  + \io (\Id+\eps n\otimes n)\nabla (c_p-\bar c)^+\cdot \nabla (c_p-\bar c)^+ \non\\
 & \mbox{}~~~~~
  + \io (\Id+\eps n\otimes n)\nabla \Phi\cdot\nabla\Big(\frac{1}{2} ((c_p-\bar c)^+)^2+\bar c (c_p-\bar c)^+\Big) = 0.\label{est:cp+}
\eea
Similarly, we get from \eqref{eq:cm}
\bea
 & \frac{1}{2} \ddt \io |(c_m-\bar c)^+|^2
  + \frac12 \io v\cdot\nabla ((c_m-\bar c)^+)^2\non\\
 & \mbox{}~~~~~
  + \io (\Id+\eps n\otimes n)\nabla (c_m-\bar c)^+\cdot \nabla (c_m-\bar c)^+ \non\\
 & \mbox{}~~~~~
  - \io (\Id+\eps n\otimes n)\nabla \Phi\cdot\nabla\Big(\frac{1}{2} ((c_m-\bar c)^+)^2+\bar c (c_m-\bar c)^+\Big) = 0.\label{est:cm+}
\eea
We now define
\be\label{def:M}
  \displaystyle
  M(r):=\begin{cases}
   0 &\textrm{ if } r\le \bar c,\\
  \frac{1}{2}((r-\bar c)^+)^2+\bar c(r-\bar c)^+ & \textrm{ if } r\ge \bar c,
  \end{cases}
\ee
Then, summing \eqref{est:cp+} and \eqref{est:cm+} and using incompressibility,
we deduce
\bea\non
  & \frac{1}{2} \ddt \io \big( |(c_p-\bar c)^+|^2+|(c_m-\bar c)^+|^2 \big)\\
  & \mbox{}~~~~~
   \le -\io (\Id+\eps n\otimes n)\nabla \Phi\cdot \nabla (M(c_p)-M(c_m)).
\eea
The integral on the right-hand side can be computed by using \eqref{eq:Phi}.
This leads to 
\bea
  \frac{1}{2} \ddt \io \big( |(c_p-\bar c)^+|^2+|(c_m-\bar c)^+|^2 \big)
   \le -\io (c_p-c_m) (M(c_p)-M(c_m)) \le 0,
\eea 
the inequality following from the monotonicity of the function $M$. 
Noting that \eqref{init} implies that the left-hand side is null
at $t=0$, we obtain the claimed estimate.
\end{proof}
\noindent%
In particular, we have obtained the additional bound
\be\label{st:21}
  \| c_p \|_{L^\infty(0,T;L^\infty(\tor))}
   + \| c_m \|_{L^\infty(0,T;L^\infty(\tor))}
   \le c.
\ee where the constant $c$ depends just on the $L^\infty$ norm of $c_p(0)$ and $c_m(0)$.
We can then test \eqref{eq:cp} by $c_p$ and \eqref{eq:cm} by
$c_m$. Using once more the positive definiteness of the 
matrix $n\otimes n$, we may note that
\be\label{co:21}
 \left|\io c_p \nabla \Phi\cdot \nabla c_p\right|
   \le \| c_p \|_{L^\infty(\tor)} \| \nabla\Phi \|_{H} \| \nabla c_p \|_{H}
   \le c \| \nabla c_p \|_{H} 
   \le c + \frac12 \| \nabla c_p \|_{H}^2,
\ee
with an analogous relation holding for $c_m$ and where the constants $c>0$ are 
independent of time in view of \eqref{st:14} and \eqref{st:21}. 
Analogously we can estimate the term $-\io\eps (n\otimes n)c_p \nabla \Phi\cdot \nabla c_p$ by \eqref{st:12}. 

Then, it is not difficult to deduce the parabolic regularity estimate
\be\label{st:22}
  \| c_p \|_{L^2(0,T;V)}
   + \| c_m \|_{L^2(0,T;V)}
   \le c.
\ee
In view of the fact that $\Phi$ is defined up to an additive constant, it is not
restrictive to assume that 
\be\label{zeromean}
  \Phi\oo = \io \Phi(t) = 0 \quad\text{for a.e.~}\,t\in (0,T).
\ee
Of course, such a normalization property, joint with \eqref{st:14},
implies
\be\label{st:14b}
  \| \Phi \|_{L^\infty(0,T;V)}
   \le c.
\ee
We have, however, a better property which is given by the following
\begin{lemma}[Uniform boundedness of $\Phi$]\label{lem:phi}
We have the additional estimate
\be\label{st:31}
  \| \Phi \|_{L^\infty(0,T;L^\infty(\tor))} \le c.
\ee 
\end{lemma}
\begin{proof}
The proof follows by applying a Moser iteration argument on equation \eqref{eq:Phi}
and using the uniform boundedness of the right-hand side following from estimate~\eqref{st:21}.
We give some highlights for the reader's convenience. As a general rule, we multiply 
equation \eqref{eq:Phi} by $(\Phi)^{p-1}:=| \Phi |^{p-1} \sign \Phi$ where the exponent $p$
will be taken larger and larger. This gives
\bea\non 
  (p-1)\io (\Id + \eps n\otimes n) | \Phi|^{p-2} \nabla \Phi \cdot \nabla \Phi 
   & = \io (c_p - c_m) | \Phi |^{p-1}\sign \Phi\\
   & \le c \io | \Phi |^{p-1} \le  c \io \left(\frac{1}{p}+\frac{p-1}{p} | \Phi |^{p}\right)\non\\
    & \le \frac{c}{p}+c\io | \Phi |^{p}.\label{co:41}
\eea
As a first step, we take $p=p_0=6$. Then, controlling the right-hand side by the Poincar\'e-Wirtinger
inequality we deduce (cf.~also \eqref{zeromean})
\be\label{co:42}
  c \io | \Phi |^{6}
   = c \| \Phi - \Phi\oo \|_{6}^6
   \le c \| \nabla \Phi \|_{2}^6
   \le c,
\ee 
the last inequality following from \eqref{st:14}. Here and below, we are noting simply by $\| \cdot \|_q$ 
the norm in $L^q(\tor)$, $1 \le q \le \infty$, for notational simplicity. We also point 
out that all the estimates obtained in this proof are uniform with respect to the time variable,
because so are \eqref{st:14} and \eqref{st:21} that serve as a starting point of the argument.

Hence, noting that
\be\label{co:43}
  (p-1)\io (\Id + \eps n\otimes n) | \Phi|^{p-2} \nabla \Phi \cdot \nabla \Phi 
  \geq\frac{4(p-1)}{p^2}\io \big| \nabla (\Phi)^{p/2} \big|^2
\ee 
at the first iteration, i.e.~for $p=6$, we deduce
\be\label{co:44}
  \big\| \nabla \Phi^3 \big\|_2 
   \le c,
\ee 
whence, recalling \eqref{st:14} and using Sobolev's embeddings,
\be\label{co:45}
  \| \Phi \|_{18}^3
  \le c(\big\| \nabla | \Phi |^3 \big\|_2 + \| \Phi \|_6^3)
   \le c.
\ee 
%
%
%

Now, in order to take care of further iterations, we need to keep trace of the 
dependence on $p$ of the various constants. Let us, then, go back to 
\eqref{co:41} with a generic $p$ and combine it with \eqref{co:43} to deduce (for $p\ge 2$)
$$
\io \big| \nabla (\Phi)^{p/2} \big|^2\le \frac{cp}{(p-1)}+\frac{cp^2}{p-1}\io |\Phi|^p\le c+c(p+2)\io |\Phi|^p
$$ where $c$ is independent of $p$.
 
Adding also $\| \Phi \|_p^p$ to both hand sides and using the Sobolev embedding, we then deduce
\bea\non
  \| \Phi \|_{3p}^p
   & = \big\| (\Phi)^{p/2} \big\|_{6}^2
   \le c \big\| (\Phi)^{p/2} \big\|_{V}^2\\
 \label{co:47}  
   & \le c \big\| (\Phi)^{p/2} \big\|_{2}^2 + c \io \big| \nabla (\Phi)^{p/2} \big|^2
   \le c+ c (p+3) \| \Phi \|_p^p\le c+cp\| \Phi \|_p^p ,
\eea
where $c$ is still  independent of~$p$. 

We define $b_p=\max (1,\| \Phi \|_p)$. Then, assuming without loss of generality that $c\ge 1$ the last inequality implies:
$$
  b_{3p}^p\le cpb_p^p
$$ 
with $c>1$ a constant independent of $p$. 
Then, since $\ln b_{3p}\le \frac{\ln(cp)}{p}+\ln b_p$, we get 
\begin{align}\nonumber
\ln b_{3^n p}&\le\frac{\ln(c 3^{(n-1)}p )}{3^{n-1} p}+\ln b_{3^{n-1} p}\\
\nonumber
&\leq \frac{\ln(c 3^{(n-1)}p)}{3^{n-1}p }+\frac{\ln(c 3^{(n-2)}p)}{3^{n-2}p }+\dots +\ln b_p.
\end{align}
and hence
$$
\ln b_{3^n p}\le \sum_{k=1}^{n-1} \frac{\ln (c3^kp)}{c3^k p}+\ln b_p.
$$
Noting that  constant $c$ is independent of $n$ and $p$ and letting $n\nearrow\infty$
we obtain \eqref{st:31}.
\end{proof}
\noindent%
It is worth observing that the bounds derived up to this point are not sufficient for 
passing to the limit in (a suitable approximation) of system \eqref{eq:cp}-\eqref{eq:n}, the main trouble being represented by the quadratic terms 
in $\nabla\Phi$ and $\nabla n$. Indeed, at the moment such quantities are 
bounded only in $L^2$ with respect to space variables. Hence, at the limit we 
might expect occurrence of defect measures. Fortunately, this is not the case, because it is possible to improve a bit the regularity properties
proved so far. 
\begin{lemma}[Additional regularity estimate]\label{lem:meyer}
 Let us assume that the initial data satisfy \eqref{hp:init}--\eqref{hp:n0}. 
 Then the following additional regularity conditions
 hold: 
 \bea\label{st:a1}
  & \| \nabla\Phi \|_{L^\infty(0,T;L^{p_M}(\tor))} \le c_{p_M}, \quad \hbox{ for some $p_M> 2$}\\
 \label{st:a2}
  & \| n_t \|_{L^{p_0}(0,T;L^{p_0}(\tor))} + \| \Delta n \|_{L^{p_0}(0,T;L^{p_0}(\tor))} \le c,\quad \hbox{ for some $p_0> 1$}\\
 \label{st:a3}
  &\| \partial \mcF (n) \|_{L^{p_0}(0,T;L^{p_0}(\tor))} \le c\quad \hbox{ for some $p_0> 1$}.
 \eea
\end{lemma}
\begin{proof}
The key point stands in the application of some refined elliptic regularity result 
to equation \eqref{eq:Phi}. Indeed, in view of the bound $|n| \le 1$ and of
the positive definiteness of $n\otimes n$, the matrix $\Id + \eps n \otimes n$
is strongly elliptic and has bounded coefficients. Since
the \rhs\  of \eqref{eq:Phi} is uniformly bounded by \eqref{mp}, we can then apply the 
integrability result \cite[Thm.~1, p.~198]{meyers},  which implies 
\be\label{st:a1temp}
\| \nabla\Phi \|_{L^\infty(0,T;L^{p_M}(\tor))} \le c_{p_M} \quad \hbox{ for some $p_M> 2$}.
\ee
Note that, at least in three space dimensions, there is no quantitative
control of $p_M$. Nevertheless, we know that  $p_M>2$. 
As a consequence of \eqref{st:a1}, \eqref{eq:n} can be rearranged in the form
\be\label{eq:n2} 
 n_t - \Delta n + \partial \mcF(n) ={\underbrace{-v\cdot\nabla n+\Omega(v) n
- D(v)n+\eps\left(\nabla\Phi\otimes \nabla \Phi\right)n}_{:=f}},
\ee 
where a simple check based on the previous estimates \eqref{st:11}, \eqref{st:12} shows 
that, at least, 
$$
v\cdot\nabla n+\Omega(v) n
- D(v)n\in L^{\frac{3}{2}}(0,T;L^{\frac{3}{2}}(\tor))
$$ which together with \eqref{st:a1temp} implies
\be\label{cond:f} 
  f \in L^{p}(0,T;L^p(\tor)).
\ee   for all $p\le p_0$ where
\begin{equation}\label{p0}
p_0:=\min \left(\frac{3}{2},\frac{p_M}{2}\right).
\end{equation}

Recalling \eqref{hp:mcF}, we observe that,  componentwise, equation \eqref{eq:n2} takes the form 
\be\label{eq:n3} 
  \partial_t n_i - \Delta n_i + F'(|n|^2) n_i = f_i,
\ee 
where $F'$ is monotone because $F$ is convex. 

This property, however, has to be a bit clarified. Indeed, the function $\mcF$ 
may be nonsmooth, and its subdifferential $\partial\mcF$ may be (and in fact
has to be, in view of assumption~\eqref{hp:F}) a singular operator. Hence, 
here and below the use of $F'$ to represent the subdifferential $\partial\mcF$ is formal 
and to make the procedure fully rigorous one should rather perform some regularization of $\mcF$ 
and then pass to the limit. Since this kind of argument is rather standard, we omit details for brevity.

Take  from now on $p=:p_0$ (for simplicity of notation). We then test \eqref{eq:n3} by 
the function $G_i(n)=|F'(|n|^2)|^{p-1} \sign F'(|n|^2) n_i$ to obtain
\bea\non
 & \frac12 \io |F'(|n|^2)|^{p-1} \sign F'(|n|^2) \ddt | n_i |^2
  + \io | F'(|n|^2) |^p n_i^2
 \\
 \label{co:n1}
 & \mbox{}~~~~~
  +{ \io |F'(|n|^2)|^{p-1} \sign F'(|n|^2)  |\nabla n_i|^2}
 + \calM_i = \io f_i \ F'(|n|^2)|^{p-1} \sign F'(|n|^2) n_i,
\eea
where the ``mixed'' term $\calM$ is given by
\bea\non
  \calM_i & = (p-1) \io |F'(|n|^2)|^{p-2} F''(|n|^2) n_i \nabla |n|^2 \cdot \nabla n_i\\
 \label{co:n2} 
  & = \frac{(p-1)}2 \io |F'(|n|^2)|^{p-2} F''(|n|^2) \nabla |n|^2 \cdot \nabla n_i^2.
\eea 
Let us sum \eqref{co:n1} for $i=1,2,3$. It is then easy to check that 
\be\label{co:n3} 
  \sum_{i=1}^3 \calM_i 
  = \frac{(p-1)}2 \io |F'(|n|^2)|^{p-2} F''(|n|^2) \nabla |n|^2 \cdot \nabla |n|^2 \ge 0
\ee 
due to convexity of $F$. 
 We split the term  $\io |F'(|n|^2)|^{p-1} \sign F'(|n|^2)  |\nabla n|^2$ over two subsets of $\tor$, namely

$$
\tor_+:=\left\{x\in\tor, |n|^2(x)\ge 1-\frac{1}{e}\right\}, \textrm{ respectively }\tor_-:=\left\{x\in\tor, |n|^2(x)< 1-\frac{1}{e}\right\},
$$
where we neglect the dependence on $t$ for simplicity. 

Then, taking into account that $F'(r)\ge 0$ for $r\in (1-\frac{1}{e},1)$, 
 neglecting the positive term $\int_{\tor_+} |F'(|n|^2)|^{p-1} \sign F'(|n|^2)  |\nabla n|^2$ on the \lhs, and using that $F'(|n|^2(x))\in (-1,0)$ for $x\in\tor_-$ 
we deduce: 
\bea\non
  & \frac12 \io |F'(|n|^2)|^{p-1} \sign F'(|n|^2) \ddt | n |^2   + \io | F'(|n|^2) |^p |n|^2\\
 \non 
  & \mbox{}~~~~~ 
  \le \io F'(|n|^2)|^{p-1} \sign F'(|n|^2) f \cdot n +\int_{\tor_-} |F'(|n|^2)|^{p-1}  |\nabla n|^2
\\
 \non
  & \mbox{}~~~~~ 
   \le \big\| F'(|n|^2)|^{p-1} \sign F'(|n|^2) \big\|_{p/(p-1)} \| f \cdot n \|_p +\io |\nabla n|^2
   \non\\
  &\mbox{}~~~~~ \le c \big\| F'(|n|^2)| \big\|_{p}^{p-1} \| f \|_p +c\non\\
 \label{co:n4} 
  & \mbox{}~~~~~ 
   \le \sigma \big\| F'(|n|^2)| \big\|_{p}^{p} + c_\sigma\| f \|_p^p +c,
\eea
where we also used H\"older's and Young's inequalities  {and the apriori bounds \eqref{st:12}}.

Now, note that 
\be\label{co:n5} 
 \frac12 \io |F'(|n|^2)|^{p-1} \sign F'(|n|^2) \ddt | n |^2
  = \ddt \io \Gamma_p(|n|^2),
\ee 
where the function $\Gamma_p$ is defined by the \rhs \ above and it is bounded from below.
Notice that $\lim_{r\to 1^-} \Gamma_p(r)<+\infty$ and that 
\be\label{co:n6} 
  \io | F'(|n|^2) |^p |n|^2
   \ge \frac 12 \io | F'(|n|^2) |^p - c
\ee
(to see this,  split the integral into the subregions 
$|n|^2\le 1/2$, where $F'$ is bounded  and $|n|\ge 1/2$ which
gives the control from below). Hence, taking $\sigma<1/2$, we see that 
the first term on the \rhs\ of \eqref{co:n4} is controlled. On the other
hand, integrating in time, we may note that the latter term
in \eqref{co:n4} is also controlled by \eqref{cond:f}.
As a consequence, we obtain  first 
\[
\|F'(|n|^2)\|_{L^p((0,T)\times\tor)}\leq c 
\]
and, as a consequence,
\[
\|\partial {\mathcal F}(n)\|_{L^p((0,T)\times\tor)}\leq c\,.
\]
Finally, comparing terms in \eqref{cp:11b} and
applying elliptic regularity results of Agmon-Douglis-Nirenberg
type, we get the bound 
\be
\| n_t \|_{L^{p}(0,T;L^{p}(\tor))} + \| \Delta n \|_{L^{p}(0,T;L^{p}(\tor))} \le c,
\ee
where we also used the regularity $n_0 \in W^{1,\frac{3}{2}}(\tor)$ which is
actually implied by our assumption~\eqref{hp:n0}. 
\end{proof}
\noindent%
In the case when the anisotropy coefficient $\eps$ is small enough compared
to the other parameters, we can prove some additional estimates.
This is stated in the following
\begin{lemma}[$H^2$-estimates]\label{lem:n}
  Let us assume that the initial data satisfy \eqref{hp:init}--\eqref{hp:n0}. 
 Furthermore, let $\eps>0$ be small enough. Then, we have 
 \be\label{st:41}
   \| \Phi \|_{L^2(0,T;H^2(\tor))} 
    + \| n \|_{L^2(0,T;H^2(\tor))} \le c.
 \ee 
\end{lemma}
\begin{proof}
We proceed in a natural way by testing \eqref{eq:n} by $-\Delta n$. 
Then, we can preliminarily observe that, by convexity of $\mcF$ 
(and consequent monotonicity of the subdifferential),
\be\label{st:51}
  - \io \partial \mcF(n)\cdot \Delta n \ge 0.
\ee 
As already noted before, this property, due to nonsmoothness of $\partial\mcF$, 
may require an approximation argument to be proved rigorously.

That said, we arrive at the bound 
\bea\non 
  & \frac12 \ddt \| \nabla n \|_H^2
   + \| \Delta n \|_{H}^2
   = \io ( v \cdot \nabla n ) \cdot \Delta n
      - \io ( \Omega(v) n ) \cdot \Delta n\\
 \label{co:51}  
  &\mbox{}~~~~~
   + \io ( D(v) n ) \cdot \Delta n
   - \io \eps (( \nabla \Phi \otimes \nabla \Phi ) n) \cdot \Delta n
   =: \sum_{j=1}^4 I_j.
\eea
and we need to estimate the terms $I_j$ on the right-hand side. A key role 
will be played by the inequality
\be\label{GN:L4}
  \| \nabla z \|_{L^4(\Omega)} 
   \le c \| z \|_{L^\infty(\Omega)}^{1/2} \| z \|_{H^2(\Omega)}^{1/2} ,
\ee
holding for every $z\in H^2(\Omega)$, $\Omega$ being a smooth bounded domain 
of $\RR^3$ (for instance $\Omega=\tor$). Then, integrating
by parts and using \eqref{eq:vincompres} with the periodic
boundary conditions, we have
\bea\non
 I_1 & = - \io ( \nabla n \odot \nabla n ): \nabla v
  \le \| \nabla n \|_{L^4(\Omega)}^2 \| \nabla v \|_{H} \\
 \non
  & \le c \| n \|_{L^\infty(\Omega)} \big( \| n \|_{H} + \| \Delta n \|_{H} \big)
         \| \nabla v \|_{H} \\
 \label{co:52}
  & \le c + \frac16 \| \Delta n \|_{H}^2 + c \| \nabla v \|_{H}^2,
\eea
where we used in an essential way the property $|n| \le 1$ almost everywhere.

Next, it is clear that
\be\label{co:53}
  I_2 + I_3 
   \le c \| n \|_{L^\infty(\Omega)} \| \nabla v \|_H \| \Delta n \|_H 
   \le \frac16 \| \Delta n \|_H^2 + c \| \nabla v \|_H^2,
\ee 
and, finally,
\bea\non
  I_4
   & \le c \eps \| n \|_{L^\infty(\Omega)} \| \nabla \Phi \|_{L^4(\Omega)}^2 \| \Delta n \|_H 
   \le c \eps \| \Phi \|_{L^\infty(\Omega)} \| \Phi \|_{H^2(\Omega)} \| \Delta n \|_H \\
 \label{co:54}  
   & \le c \eps^2 \| \Phi \|_{H^2(\Omega)}^2 + \frac16 \| \Delta n \|_H^2,
\eea  where for the last inequality we implicitly used Lemma~\ref{lem:phi}.

Taking \eqref{co:52}-\eqref{co:54} into account, \eqref{co:51} implies
\be\label{co:55}
  \ddt \| \nabla n \|_H^2
   + \| \Delta n \|_{H}^2
   \le c + c \| \nabla v \|_{H}^2
   + c\eps^2\|\Phi\|_H^2+ c \eps^2 \| \Delta\Phi \|_{H}^2,
\ee
where we point out that the constants $c$, in particular the last one, 
may depend on the various parameters of the problem, but are independent
of the coefficient $\eps$.

In order to control the last term, we apply elliptic regularity results
to \eqref{eq:Phi} (or, in other words, we test it by $-\Delta\Phi$) to
obtain
\bea\non 
  \| \Delta \Phi \|_{H}
  & \le c \big( \eps \| \nabla n \|_{L^4(\Omega)} \| n \|_{L^\infty(\Omega)} \| \nabla \Phi \|_{L^4(\Omega)} 
   + \eps \| D^2 \Phi \|_{H} \| n \|_{L^\infty(\Omega)}^2
   + \| c_p - c_m \|_H  \big)\\
 \non
  & \le c \big( \eps \| \Delta n \|_{H}^{1/2} \| \Delta \Phi \|_{H}^{1/2} \| \Phi \|_{L^\infty(\Omega)}^{1/2}
   + \eps \| \Delta \Phi \|_{H} 
   + 1 \big)\\
 \label{co:56}  
  & \le \frac14 \| \Delta n \|_{H} + c \eps \| \Delta \Phi \|_{H} + c.
\eea
where we have repeatedly used \eqref{GN:L4}. At this point, we may assume $\eps$ so small that $c \eps \le 1/2$. 
Then, the second term on the \rhs\ can be absorbed by the corresponding quantity on the \lhs. 
Squaring the resulting relation, we then deduce
\be\label{co:57}
  \| \Delta \Phi \|_{H}^2
   \le \frac14 \| \Delta n \|_{H}^2 
   + c.
\ee
Replacing into \eqref{co:55}, we arrive at
\be\label{co:58}
  \ddt \| \nabla n \|_H^2
   + \frac34\| \Delta n \|_{H}^2
   \le c + c \| \nabla v \|_{H}^2
   + c \frac{\eps^2}2 \| \Delta n \|_{H}^2,
\ee
which, possibly assuming $\eps$ small (such that $c\eps^2/2\leq 1/4$),  reduces to
\be\label{co:58bis}
  \ddt \| \nabla n \|_H^2
   + \frac12 \| \Delta n \|_{H}^2
   \le c + c \| \nabla v \|_{H}^2.
\ee
Integrating in time and recalling \eqref{st:11} we obtain the estimate
for $n$ in \eqref{st:41}. The estimate for $\Phi$ is then deduced by integrating
in time~\eqref{co:57}.

\end{proof}


\section{Weak sequential stability: proof of Theorem~\ref{thm:stab}}
\label{sec:limit}

Let us assume $(c_{p,k},c_{m,k},\Phi_k,v_k,n_k)$ to be a family of approximating
solutions complying with the estimates derived in the previous section
uniformly with respect to the parameter $k\in \NN$. We will then prove that there exists a (non-relabelled) sequence of the above sequence 
tending, in a suitable way, to a quintuple $(c_{p},c_{m},\Phi,v,n)$ solving
system~\eqref{eq:cp}-\eqref{eq:n} in the sense specified in Definition~\ref{def:weak}.

To this aim, we start deducing some convergence properties (as mentioned,
we will always assume to hold up to the extraction of subsequences)
arising as a consequence of the bounds \eqref{st:11}-\eqref{st:14}, 
\eqref{st:21}, \eqref{st:22}, \eqref{st:14b}, \eqref{st:31},  \eqref{st:a1}--\eqref{st:a3} and \eqref{st:41}.   Namely, we have
that  there exists $\lambda\in L^{p_0}(0,T;L^{p_0}(\tor))$ such that
\bea\label{li:11}
  & v_k \to v \quad\text{weakly star in }\, L^\infty(0,T;H) \cap L^2(0,T;V),\\
 \label{li:12}
  & n_k \to n \quad\text{weakly star in }\, L^\infty(0,T;V) \cap L^\infty(0,T;L^\infty(\tor)),\\
 \label{li:13}
  & \Phi_k \to \Phi \quad\text{weakly star in }\, L^\infty(0,T;V) \cap L^\infty(0,T;L^\infty(\tor)),\\
 \label{li:14}
  & c_{p,k},~c_{m,k} \to c_p,~c_m \quad\text{weakly star in }\, L^2(0,T;V) \cap L^\infty(0,T;L^\infty(\tor)),
  \\
  \label{li:15}
  & \nabla\Phi_k\to \nabla \Phi\quad\text{weakly star in }L^\infty(0,T;L^{p_M}(\tor)),
  \\
  \label{li:16}
  &  \partial_t n_k, \, \Delta n_k,\,  \partial \mcF(n_k)\to n_t, \, \Delta n, \, \lambda\quad\text{weakly in }L^{p_0}(0,T;L^{p_0}(\tor)),
\eea
where in deducing \eqref{li:13} we also used the normalization $(\Phi_k)\oo = 0$ and $p_0,p_M$ are the exponents 
introduced in Lemma~\ref{lem:meyer}. 
Of course this implies in particular $\Phi\oo = 0$. Let us notice that, in the limit, we preserve
the boundedness conditions $0\le c_p\le \bar c$, $0\le c_m\le \bar c$, $|n| \le 1$
almost everywhere in~$\tor$.
In addition to that,  if $\eps$ is sufficiently small (cf.~Lemma~\ref{lem:n}), we also get:
\be\label{li:17}
\Phi_k, \, n_k\to \Phi, \, n \quad\text{weakly in }L^2(0,t;H^2(\tor)).
\ee
In the following {we show how to treat the passing to the limit just for the most difficult terms}. We first note that, by \eqref{st:11} and interpolation, 
\be\label{li:21}
  \| v_k \|_{L^{4}(0,T;L^{3}(\tor))} \le c,
\ee
whence, using \eqref{st:22}, there follows 
\be\label{li:22}
  \| v_k\cdot \nabla c_{p,k} \|_{L^{4/3}(0,T;L^{6/5}(\tor))} 
   + \| v_k\cdot \nabla c_{m,k} \|_{L^{4/3}(0,T;L^{6/5}(\tor))} 
  \le c.
\ee
Then, using uniform boundedness of $c_{p,k}$, $c_{m,k}$  { as well as the bounds \eqref{st:14}, \eqref{st:21}}
it is not difficult to deduce from \eqref{eq:cp}, \eqref{eq:cm} that{
\be\label{li:23}
  \| \de_t c_{p,k} \|_{L^{4/3}(0,T;V')} + \| \de_t c_{m,k} \|_{L^{4/3}(0,T;V')} \le c.
\ee
}
Hence, {taking also into account \eqref{st:a2},} the Aubin-Lions lemma with the uniform boundedness property gives 
\be\label{li:24}
  c_{p,k},~c_{m,k},~n_k \to c_p,~c_m,~n \quad\text{strongly in }\, L^q(0,T;L^q(\tor))
   ~~\forall\,q\in[1,\infty).
\ee

Then, using \eqref{li:16}, \eqref{li:24}, the monotonicity  of 
$\partial\mcF$, and the result \cite[Prop.~1.1, p.~42]{barbu}, we get $\lambda=\partial\mcF(n)$. 
Moreover, by \eqref{li:13} and \eqref{li:14} we get
\be
\|c_{p,k}\nabla\Phi_k\|_{L^\infty(0,T;H)}+\|c_{m,k}\nabla\Phi_k\|_{L^\infty(0,T;H)}  \leq c, \non
\ee
whence 
\be 
c_{p,k}\nabla\Phi_k\to c_p\nabla \Phi, \quad c_{m,k}\nabla\Phi_k\to c_m\nabla \Phi  \quad\text{weakly star in }L^\infty(0,T;H),\non
\ee
where we have used also \eqref{li:24}. Using the Gagliardo-Nirenberg inequality (cf.~\cite{nire}) together with \eqref{li:16} and the fact that $|n_k|\leq 1$, 
we get 
\be
\|\nabla n_k\odot\nabla n_k\|_{L^s(0,T;L^s(\tor))}\leq c\quad\text{for some exponent $s>1$},\non
\ee
Finally, using the bound on $\partial_tn_k$ in \eqref{li:16} and again  the Gagliardo-Nirenberg inequality (cf.~\cite{nire}) interpolating between the spaces $L^\infty(0,T; L^\infty(\tor))$ and $L^{p_0}(0,T; W^{2,p_0}(\tor))$ at place $1/2$, we also get the convergence
\be
\nabla n_k\odot\nabla n_k\to \nabla n\odot\nabla n \quad \text{weakly in }L^s(0,T;L^s(\tor)),
\ee 
which is sufficient in order to conclude the passage to the limit as $k\to\infty$ in order to obtain the claimed weak solutions.


\begin{appendix}

\section*{Acknowledgements}

ER and GS were partially supported by GNAMPA (Gruppo Nazionale per l'Analisi Matematica, la Probabilit\`a e le loro Applicazioni)
of INdAM (Istituto Nazionale di Alta Matematica). This research has been performed in the framework of the project Fondazione Cariplo-Regione Lombardia MEGAsTAR
``Matema\-tica d'Eccellenza in biologia ed ingegneria come acceleratore di una nuova strateGia per l'ATtRattivit\`a dell'ateneo pavese''.
This research was also supported by the Italian Ministry of Education, University and Research (MIUR): Dipartimenti di Eccellenza Program (2018--2022) -
Dept. of Mathematics ``F. Casorati'', University of Pavia. The work of A.Z. is supported by the Basque Government through the BERC 2018-2021
program, by Spanish Ministry of Economy and Competitiveness MINECO through BCAM
Severo Ochoa excellence accreditation SEV-2017-0718 and through project MTM2017-82184-R
funded by (AEI/FEDER, UE) and acronym ``DESFLU''. The research of E.F. and V.M. leading to these results has received 
funding from the
Czech Sciences Foundation (GA\v CR), Grant Agreement
18--05974S. The Institute of Mathematics of the Academy of Sciences of
the Czech Republic is supported by RVO:67985840.

\bigskip

\end{appendix}


\end{document}